\documentclass{amsart}
\usepackage{amsmath}
\usepackage{amsrefs}
\usepackage{amssymb}
\usepackage[all]{xy}
\newcommand{\cC}{\mathcal{C}}
\newcommand{\CC}{\mathbb{C}}
\newcommand{\ZZ}{\mathbb{Z}}
\newcommand{\NN}{\mathbb{N}}
\newcommand{\cL}{\mathcal{L}}
\newcommand{\g}{\mathfrak{g}}
\newcommand{\End}[1]{\mathrm{End}(#1)}
\newcommand{\Hom}[2]{\mathrm{Hom}(#1,#2)}
\usepackage{a4wide}
\newtheorem{thm}{Theorem}[section]
\newtheorem{cor}[thm]{Corollary}
\newtheorem{exa}[thm]{Examples}
\newtheorem{lem}[thm]{Lemma}
\newtheorem{defn}[thm]{Definition}
\newtheorem{prop}[thm]{Proposition}
\newtheorem{question}[thm]{Question}

\title{On the notion of 'retractable modules' in the context of algebras.}
\author{Christian Lomp}
\address{Department of Mathematics, FCUP, University of Porto, Rua Campo Alegre 687, 4169-007 Porto, Portugal}
\email{clomp@fc.up.pt}
\dedicatory{dedicated to Patrick and John on the occasion of their 70th birthdays}
\thanks{This research was funded by the European Regional Development Fund through the programme COMPETE and by the Portuguese Government through the FCT - Funda\c{c}\~{a}o para a Ci\^encia e a Tecnologia under the project PEst-C/MAT/UI0144/2011.}

\begin{document}
\maketitle
\begin{abstract}
This is a survey on the usage of the module theoretic notion of a ``retractable module" in the study of algebras with actions. We explain how classical results can be interpreted using module theory and end the paper with some open questions.
\end{abstract}

\section{Introduction}
This note is written for module theorists and intends to show where the module theoretic notion of a retractable module plays a role in the context of algebras with certain additional structure. These "additional structures" include group actions, involutions, Lie algebra actions or more generally Hopf algebra actions as well as the bimodule structure of the algebra (and combinations of all these).
Such algebra $A$ is usually a subalgebra of a larger algebra $B$ and has the structure of a cyclic left $B$-module, while its endomorphism ring $\mathrm{End}_B(A)$ is isomorphic to a subalgebra $A^B$ of $A$. In this intrinsic situation the condition on $A$ to be a retractable $B$-module means that the subring $A^B$ has non-zero intersection with all non-zero left ideals of $A$ that are stable under the module action of $B$. We will first recall the module theoretical notion of a retractable module and set it in a categorical and lattice theoretical context. In the second section we will examine various situations of algebras $A$ with additional structures and recall many classical theorems that can be expressed in terms of module theory. The last section deals with open problems around retractable modules in the context of algebras. Note that all rings/algebras are considered to be associative and unital. Modules are usually meant to be left modules and homomorphisms are acting from the right.

\section{Module Theory}

A retractable module is a (left) $A$-module $M$, over some ring $A$, such that there exist non-zero homomorphisms from $M$ into each of its non-zero submodules. The notion of a retractable module appeared first in the work of Khuri \cite{Khuri5} and had  since then been used in connection with  primness conditions and the nonsingularity of a module and its endomorphism ring  (see \cite{Khuri,Khuri2,Khuri3,Khuri4,Zelmanowitz_closed}).
 One of Khuri's result is the establishment of a bijective correspondence between closed submodules of a module $M$ and closed left ideals of its endomorphism ring $S=\mathrm{End}_A(M)$ in case $M$ is non-degenerated (see \cite{Khuri0, Zelmanowitz_closed, Zhou}). A module is non-degenerated if its standard Morita context is non-degenerated. Recall that a Morita context between two rings $A$ and $S$ is a quadruple $(A,M,N,S)$ where ${_AM}_S$ 
 and ${_SN}_A$ are bimodules with bimodule maps $(-,-): M {\otimes_S} N \rightarrow A$ and $[-,-]:N{\otimes_A} M \rightarrow S$ satisfying 
$m[n,m'] = (m,n)m'$ and $n(m,n') = [n,m]n'$ for all $m,m'\in M$non-degenerated and $n,n'\in N$. The context is called non-degenerated if $M_S$ is faithful and for all $0\neq m \in M$ also $[N,m]\neq 0$. The standard Morita context of a left $A$-module $M$ is the context $(A,M,M^\ast, S)$ with $S=\mathrm{End}_A(M)$ and $M^\ast=\mathrm{Hom}_A(M,A)$ and the maps
\begin{equation}
(-,-): M {\otimes_S}  M^\ast  \rightarrow A \qquad (m,f) := (m)f \qquad \forall m\in M, f\in M^\ast
\end{equation}
\begin{equation}
[-,-]: M^\ast {\otimes_A} M \rightarrow S \qquad [f,m]  := [n \mapsto (n)f m] \qquad \forall m\in M, f\in M^\ast
\end{equation}
$M_S$ is obviously always faithful and $M$ is non-degenerated if and only if $[M^*, m] \neq 0$ for all $0\neq m\in M$. In this case there exist for any non-zero submodule $N$ of $M$  and non-zero element $m\in N$ a homomorphism $f:M\rightarrow A$ such that the map $\tilde{f}:M\rightarrow N$ with 
$(n)\tilde{f} = (n)fm\in Am\subseteq N$, for $n\in M$, is non-zero. Hence any non-degenerated module is in particular retractable.

Retractable modules have gained recently further attention in \cite{SmithVedadi, SmithVedadi2, HaghanyVedadi, HaghanyKaramzadehVedadi, lomp_prime,TolooeiVedadi, Tamer}, but have previously also played a major role in the context of algebras. The theorems of Bergman-Isaacs or Rowen say that in certain situations an algebra $A$ with a group action $G$ or considered as bimodule is a retractable module considered over the skew group ring $A\ast G$ or over the enveloping algebra $A^e=A\otimes A^{op}$. In case $\partial$ is a derivation on an algebra $A$, then $A$ is a retractable $A[x;\partial]$-module if $\partial$ is locally nilpotent. Furthermore Cohen's question raises the problem as to whether a semiprime algebra $A$ with an action of a semisimple Hopf algebra $H$ is a retractable $A\# H$-module.
With this in mind, the survey was written to illustrate the use of the module theoretic notion of retractability in the context of algebras.

\subsection{Categorical notions}
A retractable module is clearly a generalisation of a self-generator. Here we shortly review this notion in the context of category theory.

\begin{defn}
Let $\cC$ be a category. An object $X$ of $\cC$ is \emph{generated} by an object $G$ of $\cC$ if 
for every pair  of distinct morphisms $f, g : X \rightarrow Y$ in $\cC$ there exists a morphism $h: G \rightarrow X$ with $hf\neq hg$.
\end{defn}
In particular if $\cC$ is an Abelian category and $X$ is not the zero object, then $\mathrm{Mor}(G,X)\neq \{0\}$, because for the identity $f=id_X$ and the zero morphism $g=0$, there exist a morphism $h:G\rightarrow X$ such that $h\neq 0$. Having this in mind the definition of a retractable object seems to be a direct generalisation of a generator.

\begin{defn}
An object $M$ of an Abelian category $\cC$ is called \emph{retractable} if $\mathrm{Mor}(M,N)\neq \{0\}$ for all subobjects $N$ of $M$, different from the zero object.
\end{defn}

Ler  $\cC$ be an Abelian category with arbitrary coproducts. Let $M$ be any object in $\cC$ and $N$ a subobject of it.
Then there exists a unique subobject $\mathrm{Tr}(M,N)$ of  $N$ such that every morphism $f:M\rightarrow N$ factors through $\mathrm{Tr}(M,N)$.
Suppose that $M$ is a retractable object, then $\mathrm{Tr}(M,N)$ is essential in $N$ for each non-zero subobject $N\in \cL$ in the sense that for all non-zero subobjects $K$  of $N$ the meet  $K\cap \mathrm{Tr}(M,N)$ is non-zero (as any $f:M\rightarrow K$ can be considered a morphism $f:M\rightarrow N$ and hence factored through $\mathrm{Tr}(M,N)$). This means in the case of a module category $\cC$, that a module $M$ is retractable if and only if for all submodules $N$ of $M$, the trace $\mathrm{Tr}(M,N)$, which is the sums of images of all homomorphisms $f:M\rightarrow N$, is essential in $N$. Loosely speaking every submodule of a  retractable module $M$ can be "approximated" by an $M$-generated submodule.

\subsection{Lattice theoretical meaning}
Let $R$ be ring and $M$ a left $R$-module with endomorphism ring $S=\End{M}$.
To link module theoretical properties of $M$ with properties of $S$ one can use the following map from the lattice $\cL(M)$ of left $R$-submodules of $M$ to  the lattice $\cL(S)$ of left ideals of $S$:

$$ \Hom{M}{-}:\cL(M)\rightarrow \cL(S) \qquad N \mapsto \Hom{M}{N}=\{f\in S\mid (M)f\subseteq N\}.$$

This map $\Hom{M}{-}$ is always a homomorphism of semilattices between $(\cL(M),\cap)$ and $(\cL(S),\cap)$ since $\Hom{M}{N\cap L}=\Hom{M}{N}\cap \Hom{M}{L}$ holds for all $N,L\in \cL(M)$. Call a homomorphism $\varphi:\cL\rightarrow \cL'$ of semilattice with least element $0$ faithful if $\varphi(x)=0 \Rightarrow x=0$. The following Lemma can be proven easily:

\begin{lem} Let $M$ be a left $A$-module.
\begin{enumerate}
\item $\Hom{M}{-}$ is injective if and only if $M$ is a self-generator.
\item $\Hom{M}{-}$ is faithful if and only if $M$ is a retractable module.
\end{enumerate}
\end{lem}

While the injectivity or faithfulness of $\Hom{M}{-}$ has to do with $M$ being a generator or retractable, the surjectivity of $\Hom{M}{-}$ deals with the projectivity of $M$ (we refer the reader to \cite{wisbauer96} for all undefined notion.):

\begin{lem} Let $M$ be a left $A$-module and $S=\mathrm{End}_A(M)$.
\begin{enumerate}
\item All cyclic left ideals of $S$ belong to the range of $\Hom{M}{-}$ if and only if $M$ is semi-projective.
\item All finitely generated left ideals of $S$ belong to the range of $\Hom{M}{-}$ if and only if $M$ is 
intrinsically-projective.
\item $\Hom{M}{-}$ is surjective if $M$ is $\Sigma$-self-projective, i.e. projective in the Wisbauer category $\sigma[M]$.
\end{enumerate}
\end{lem}

Considering semi-projective retractable modules combines the light self-generator and self-projectivity condition on $M$. Such modules were studied for example in \cite{HaghanyVedadi}.

\section{Algebras with additional structures}
An (associative, unital) algebra $A$ over a commutative ring $R$ is an $R$-module $A$ such that there exist $R$-linear maps
$$\mu: A{\otimes_R} A\rightarrow A \qquad \mbox{ and } \qquad \eta:R\rightarrow A,$$
called the multiplication of $A$ and unit of $A$ respectively, such that $A$ with $\mu$ as multiplication, defined as $ab=\mu(a\otimes b)$ for $a,b\in A$, and $1_A=\eta(1)$ as unit element forms an associative and unital  ring. 

It is easy to see that by taking  $R=\ZZ$, any (associative, unital) ring is an (associative, unital) algebra over $\ZZ$.  Thus for those that do not like the idea of $R$-algebras, they might for the beginning just ignore $R$ and think of $A$ being and ordinary ring.
Clearly $\eta$ does not need to be injective, just think of $A=\ZZ_n [x]$, for some $n>1$, which is a $\ZZ$-algebra and $\eta:\ZZ\rightarrow \ZZ_n \subseteq \ZZ_n [x]$ is the canonical map, where $\ZZ_n = \ZZ/n\ZZ$. Moreover the image of $\eta$ lies always in the centre of $A$ and in particular $A$ is an $R'$-algebra for any subring $R'$ of the centre $Z(A)=\{a\in A\mid ab=ba\;\; \forall b\in A\}.$ In the following let $R$ be always a commutative ring and $A$ an $R$-algebra.

\subsection{Algebras that are retractable as bimodule}\label{bimodule}
The endomorphism ring $\mathrm{End}_R(A)$ of $A$ as $R$-module is itself an $R$-algebra whose $R$-module structure is given as follows: for all $r\in R, f\in \mathrm{End}_R(A)$ set $r\cdot f:A\rightarrow A$ by $(r\cdot f)(x):=rf(x)$ for all $x\in A$. The multiplication of $\mathrm{End}_R(A)$ is given by the composition of functions and the unit  map is given by $\eta:R\rightarrow \mathrm{End}_R(A)$ sending $r\mapsto r\cdot \mathrm{id}_A$. 

For each element $a\in A$ there are two $R$-linear maps of $A$ which are the left and the right multiplication by $a$:
$$l_a: A\rightarrow A \qquad l_a(x):= ax \qquad \forall x\in A,$$
$$r_a: A\rightarrow A \qquad r_a(x):= xa \qquad \forall x\in A.$$
Note that since the multiplication of $A$ is supposed to be associative,  $l_a$ and $r_b$ commute, i.e. $l_a\circ r_b=r_b\circ l_a$ in $\mathrm{End}_R(A)$, for any $a,b\in A$. The subalgebra of $\mathrm{End}_R(A)$ generated by all maps $l_a$ and $r_b$ for $a,b\in A$. Is called the multiplication algebra of $A$ and denoted by $M(A)$. 

Left $M(A)$-modules $M$ can be considered as bimodules over $A$, where one defines for all $a,b\in A$ and $m\in M$:
$$ a m := l_a\bullet m \qquad \mbox{ and } \qquad m b := r_b\bullet m.$$
The bimodule compatibility condition $(am)b = a (mb)$ for all $m\in M$ holds, because of 
$(r_b\circ l_a -  l_a\circ r_b)\bullet M = 0$. Analogously any $A$-bimodule has a natural structure as left $M(A)$-module given by $l_a\bullet m = am$ and $r_b\bullet m = mb$, for $a,b\in A$ and $m\in M$.
The enveloping algebra $A^e=A {\otimes_R} A^{op}$ is also an $R$-algebra, where $A^{op}$ denotes the opposite ring of $A$. The multiplication of $A^e$ is defined as 
$$(a\otimes x)(b\otimes y) := (ab) \otimes (yx) \qquad \forall a,b,x,y \in A.$$
Moreover the map $\psi: A^e \rightarrow \mathrm{End}_R(A)$ given by  $$ a\otimes b \mapsto l_a\circ r_b \qquad \forall a,b\in A$$
is a surjective algebra map from $A^e$ to $M(A)$ whose kernel is the annihilator of $A$, where $A$ is naturally considered a left $A^e$-module by $(a\otimes b) \bullet x = axb$ for all $a,b,x\in A$. Hence 
$$\mathrm{Ker}(\psi) = \left\{ \sum_{i=1}^n a_i\otimes b_i \in A^e \mid \sum_{i=1}^n a_ixb_i = 0 \qquad \forall x\in A\right\} = \mathrm{Ann}_{A^e}(A).$$

For a bimodule $M\in  A^e$-Mod one defines its centre as  $$ Z(M) = \{ m\in M \mid am=ma \:\:\forall a\in A\}.$$
The canonical map 
\begin{equation}\label{psi_bi}
\psi_M: \mathrm{Hom}_{A^e}(A,M) \longrightarrow Z(M) \qquad \mbox{ given by } \qquad f \mapsto  (1)f \qquad \forall f\in\mathrm{Hom}_{A^e}(A,M).
\end{equation}
is a bijection, where the left $A^e$-module homomorphism is applied from the right. The inverse of this map is 
\begin{equation}
\psi_M^{-1}: Z(M) \longrightarrow \mathrm{Hom}_{A^e}(A,M) \qquad \mbox{ given by } \qquad m \mapsto [a\mapsto a\cdot m] \qquad \forall m\in Z(M).
\end{equation}
In particular $\psi_A:\mathrm{End}_{A^e}(A) \simeq Z(A)$ is an isomorphism of $R$-algebras and the bijections $\psi_M$ are left $Z(A)$-linear maps.

\begin{lem} $A$ is a retractable left $A^e$-module if and only if $A$ has a large centre $Z(A)$, i.e. every non-zero two-sided ideal of $A$ contains a non-zero central element.
\end{lem}
\begin{proof} This follows from the the fact that the $A^e$-submodules of $A$ are precisely the two-sided ideals $I$ and from the bijection
$$\psi_I: \mathrm{Hom}_{A^e}(A,I) \simeq Z(I) = I\cap Z(A).$$
\end{proof}

 There are at least two important results that have to be mentioned in this context. The first is due to Rowen and says the following (see \cite{Rowen}):

\begin{thm}[Rowen, 1972]\label{RowenTheorem}
Any semiprime PI algebra has a large centre.
\end{thm}
Recall that an $R$-algebra $A$ is a PI-algebra if it there exists an element  $f(x_1,\ldots, x_n)$ in the free algebra $R\langle x_1,\ldots, x_n\rangle$ over $R$ such that $f(a_1,\ldots, a_n)=0$ for any substitution $a_1,\ldots, a_n\in A$ and such that one of the coefficients of a monomial of highest degree of $f$ is $1$. Examples of semiprime PI-algebras are matrix algebras over division algebras that are finite dimensional over their centre or more generally any semiprime algebra that is a finitely generated when considered a module over its centre. Thus Rowen's theorem says that any semiprime PI-algebra is a retractable left $A^e$-module. For a non-trivial example one might consider the quantum plane at root of unity. Let $q\in \CC\setminus\{0\}$. The quantum plane over $\CC$ with parameter $q$ is the algebra 
$$A=\CC_q[x,y] = \CC\langle x,y\rangle/\langle yx-qxy\rangle.$$
Elements of $A$ can be uniquely written as linear combinations of monomials of the form $x^iy^j$ for $i,j\geq 0$. The relation $yx=qxy$ makes $\CC_q[x,y]$ a non-commutative algebra if $q\neq 1$. An elementary calculation shows that the centre of $\CC_q[x,y]$ is $Z(A)=\CC$ if $q$ is not a root of unity and that it is $Z(A)=\CC[x^n,y^n]$ if $q$ is a primitive $n$-th root of unity. In the later case $A$ is generated by all monomials of the form $x^iy^j$ with $0\leq i,j<n$ as a module over $Z(A)$. Hence $A$ is a PI-algebra. Since $A$ is also a domain the centre is large, i.e. any non-zero ideal of $\CC_q[x,y]$ contains a non-zero polynomial of the form $f(x^n,y^n)$.

The second result in this context is Puczy\l owski and Smoktunowicz' description of the Brown-McCoy radical of an algebra $A$ from \cite{Edmund-Agata}. Recall that the Brown-McCoy radical $\mathrm{BMc}(A)$ of $A$ is the intersection of all maximal two-sided ideals. This means that the Brown-McCoy radical is the module theoretic radical of $A$ as bimodule, i.e. $\mathrm{BMc}(A) = \mathrm{Rad}({_{A^e}A})$.
Puczy\l owski and Smoktunowicz described the Brown-McCoy radical of $A[x]$ using the following ideal:
$$PS(A)=\bigcap \left\{ P \subseteq A \mid P \mbox{ is a prime ideal and $A/P$ has a large centre}\right\}.$$

\begin{thm}[Puczy\l owski-Smoktunowicz, 1998]
$\mathrm{BMc}(A[x])=PS(A)[x]$.
\end{thm}

Their result relies on the following (surprising) Lemma from \cite{Edmund-Agata}:
\begin{lem} Let $M$ be a maximal ideal of $A[x]$. If $A\cap M=0$, then $A$ has a large centre.
\end{lem}
In the case of the Lemma, if such maximal ideal $M$ of $A[x]$ exist with $M\cap A=0$, then $A$ will also be a prime ring. Recall that a module $M$ over some ring $A$ is called \emph{compressible} if $M$ embeds into any non-zero submodule of it, i.e. for any $0\neq N\subseteq M$ there exists an injective $A$-linear map $f:M\rightarrow N$. Prime algebras with large centre are precisely the algebras that are compressible as bimodule.
\begin{lem}[{see \cite[35.10]{wisbauer96}}] An algebra $A$ is a compressible $A^e$-module if and only if it is prime and has a large centre.
\end{lem}
Hence, in module theoretic terms, $PS(A)$ is the intersection of all those $A^e$-submodules $P$ of $A$ such that $A/P$ is a compressible $A^e$-module.

\subsection{Derivations}

A ($R$-linear) derivation of an $R$-algebra $A$ is an $R$-linear map $\partial: A\rightarrow A$ such that $\partial(ab)=\partial(a)b+a\partial(b)$ for all $a,b\in A$.  Examples are ordinary partial derivations $\partial_{x_i}$ on the polynomial ring $R[x_1, \ldots, x_n]$ over $R$. For any $a\in A$ of an $R$-algebra $A$, its commutator  $\partial(a)=[a,-]$, with $[a,b]=ab-ba$ for $b\in A$,  is a derivation, called an  {\emph{inner derivation}} of $A$.

Given a derivation $\partial$ one constructs the differential operator ring $B=A[x;\partial]$ as the $R$-algebra generated by $A$ and $x$ subject to 
$$ xa=ax + \partial(a) \qquad \forall a\in A.$$
The algebra $A[x;\partial]$ can be constructed as a subalgebra of $\mathrm{End}_{A}(A[x])$ such that $A[x;\partial]$ is  a free left $A$-module with basis $\{ x^i \mid i\in \NN\}$. Hence the elements of $B$ can be uniquely written as (left) polynomials $\sum_{i=0}^n a_ix^i$ with $a_i\in A$.
Moreover $A$ becomes a left $A[x;\partial]$-module with respect to the action  $x\cdot a = \partial(a)$ or more generally
$$\left( \sum_{i=0}^n a_ix^i \right) \cdot b = \sum_{i=0}^n a_i \delta^i(b) \qquad \forall b\in A, \:\: \forall \sum_{i=0}^n a_ix^i \in B.$$
The left $A[x;\partial]$-submodules of $A$ are precisely those left ideals $I$  of $A$ that are stable under the derivation, i.e. $\partial(I)\subseteq I$.
For any left $A[x;\partial]$-module $M$ one defines its submodule of constants as
$$M^\partial = \{m\in M \mid x\cdot  m=0\} = \mathrm{Ann}_M(x).$$
For $M=A$ one has $A^\partial=\mathrm{Ker}(\partial)$ which is easily seen to be  a subalgebra of $A$.
Analogously to the bimodule situation we have the following $R$-linear isomorphisms for any left $A[x;\partial]$-module $M$:
\begin{equation}\label{Iso_Derivation}
\psi_M: \mathrm{Hom}_{A[x;\partial]}(A,M) \longrightarrow M^\partial \qquad \mbox{ given by } \qquad f \mapsto  (1)f \qquad \forall f\in\mathrm{Hom}_{A[x;\partial]}(A,M).\end{equation}
Its inverse map is 
\begin{equation}\psi_M^{-1}: M^\partial \longrightarrow \mathrm{Hom}_{A[x;\partial]}(A,M) \qquad \mbox{ given by } \qquad m \mapsto [a\mapsto a\cdot m] \qquad \forall m\in M^\partial.\end{equation}
In particular $\psi_A:\mathrm{End}_{A[x;\partial]}(A) \simeq A^\partial$ is an isomorphism of $R$-algebras and the bijections $\psi_M$ are left $A^\partial$-linear maps. Using these isomorphisms the following Lemma is obvious:

\begin{lem}
$A$ is a retractable $A[x;\partial]$-module if and only if $A^\partial$ is large in $A$, i.e. $A^\partial$ intersects all non-trivial $\partial$-stable left ideals of $A$ non-trivially.
\end{lem}

A sufficient condition for this to happen is the local nilpotency of  $\partial$, i.e. if for every $a\in A$, there exists $n\in \NN$ such that $\partial^n(a)=0$. 

\begin{prop} If  $\partial$ is locally nilpotent, then $A$ is a retractable $A[x;\partial]$-module.
\end{prop}
\begin{proof} The proof of this fact is obvious, because if $0\neq a\in I$ is a non-zero element of an $\partial$-stable left ideal $I$ of $A$, then  by hypothesis there exists $n\in \NN$ such that $\partial^n(a)=0$. Take the least $n\in \NN$ such that  $\partial^n(a)=0$, then $\partial^{n-1}(a)$ is a non-zero element of $I\cap A^\partial$, which proves that $A^\partial$ is large in $A$ and hence $A$ is a retractable $A[x;\partial]$-module.
\end{proof}
For example the partial derivatives $\frac{\partial}{\partial x_i}$  of $A=R[x_1, \ldots, x_n]$ for any $i=1,\ldots, n$ are locally nilpotent. However it is unknown when $A$ is a retractable $A[x;\partial]$-module for an arbitrary derivation $\partial$.

\begin{question} What are necessary and sufficient conditions for $A$ to be a retractable $A[x;\partial]$-module? in other words, find conditions on $A$ and $\partial$ such that any non-zero $\partial$-stable left ideal contains a non-zero constant.
\end{question}

Zelmanowitz called a left $R$-module \emph{fully retractable} if  for any non-zero submodule $N$ and non-zero $g:N \rightarrow M$ there exists $h:M\rightarrow N$ such that $hg\neq 0$. It is not clear when $A$ is fully retractable as $A[x;\partial]$-module. The next Proposition can be found in \cite{BorgesLomp}.

\begin{prop}[Borges-Lomp, 2011] Let $\partial$ be a locally nilpotent derivation of $A$.
\begin{enumerate}
\item $A$ is a compressible left $A[x;\partial]$-module, provided $A^\partial$ is a domain.
\item  $A^\partial$ is a left Ore domain if and only if $A$ is a uniform left $A[x;\partial]$-module and $A^\partial$ is a domain.
\item  $A$ is critically compressible as left $A[x;\partial]$-module if and only if $A^\partial$ is a left Ore domain and $A$ is fully retractable as left $A[x;\partial]$-module.
\end{enumerate}
\end{prop}

Let $R=k$ be a field of characteristic zero and $\partial$ a  locally nilpotent derivation on $A$ such that there exists an element $a\in A$ with $\partial(a)=1$. Then by \cite[Proposition 3.10]{BorgesLomp} $A$ is a self-projective $A[x;\partial]$-module and module theory yields the following result (see \cite[Proposition 3.12]{BorgesLomp})

\begin{prop}[Borges-Lomp, 2011]
Let $A$ be an algebra over a field $k$ of characteristic zero and $\partial$ a locally nilpotent derivation of $A$ such that $\partial(a)=1$, for some $a\in A$. Then the following statements are equivalent:
\begin{enumerate}
\item[(a)] $A^\partial$ is a left Ore domain;
\item[(b)] $A$ is a left Ore domain;
\item[(c)] $A$ is a critically compressible $A[x; \partial]$-module.
\end{enumerate}
\end{prop}

\begin{exa}[Goodearl, 1980]
Let $A=k[[t]]$ be the power series ring over a field $k$ of characteristic $0$ and let $\partial$ be the derivation with $\partial(t^n)=nt^n$ for all $n\geq 0$. Certainly $\partial$ is not locally nilpotent.
Any ideal of $A$ is $\partial$-stable, because the proper ideals are of the form $I=At^n$ for $n\geq 0$.
Since for any $a=\sum_{n=0}^\infty a_n t^n \in A$ one has $\partial(a) =\sum_{n=0}^\infty n a_n t^n \neq 1$ 
we have that there does not exist any $a\in A$ with $\partial(a)=1$. Nevertheless $A$ is a self-projective left $A[x;\partial]$-module. 
To see this note that by the correspondence (\ref{Iso_Derivation}) it is enough to show that $(A/I)^\partial = (A^\partial + I)/I$ for any $\partial$-stable left ideal $I$ of $A$. Let  $I=At^m$ be any ideal of $A$ and $a=\sum_{n=0}^\infty a_n t^n$ with $\partial(a)\in I$. Then there exists $b\in A$ such that 
 $\partial(a)=\sum_{n=0}^\infty na_n t^n = bt^m\in I$, which implies that $a_i=0$ for all $1\leq i < m$. Thus $a=a_0+b't^m$ for some $b'\in A$ and 
 $a+I=a_0+I$ in $A/I$. Since $a_0\in A^\partial$ this shows 
$(A/I)^\partial \subseteq (A^\partial + I)/I$ while the reversed inclusion is obvious. Since $A$ is a Noetherian integral domain, $A[x;\partial]$ is a left Noetherian Ore domain. However as $A^\partial=k$ is the base field and $A$ is not simple as left $A[x;\partial]$-module, $A$ is not retractable and hence not compressible as $A[x;\partial]$-module.
\end{exa}

Note that the set $\mathrm{Der}_R(A)$ of derivations on the $R$-algebra $A$ forms a Lie algebra with the ordinary Lie product induced by the product(=composition) of $\mathrm{End}_R(A)$, i.e. if $\partial, \partial \in \mathrm{Der}_R(A)$, then their commutator 
$$[\partial, \delta] = \partial\circ \delta - \delta \circ \partial \in \mathrm{Der}_R(A)$$
is again a derivation of $A$. An action of an arbitrary abstract Lie algebra $\mathfrak{g}$ over $R$ by derivations is given by a map of Lie algebras $d:\mathfrak{g}\rightarrow \mathrm{Der}_R(A)$. We shall write the image of $x\in\mathfrak{g}$ under $d$ as $d_x$. An analogous construction to the construction of the differential operator ring is  given by a new product on the tensor product of $A$ and the universal enveloping algebra $U(\mathfrak{g})$ of $\mathfrak{g}$. The new algebra is called the smash product of $A$ and $U(\mathfrak{g})$ and is denoted by $A\# U(\mathfrak{g})$. The product is determined by
$$ (1\# x)(a\# 1) = a \# x + d_x(a) \# 1 \qquad \forall x\in \mathfrak{g}, a\in A.$$
Later we will shortly mention Hopf algebras and their action on rings and $U(\mathfrak{g})$ is one of the examples. 
Again $A$ becomes a left $A\# U(\mathfrak{g})$-module where the module action is given by $(a\# x) \cdot b = a \: d_x(b)$ for all $a,b\in A$ and $x\in \mathfrak{g}$. Again one can consider the subset of all those elements $a\in A$ such that $d_x(a)=0$ for all $x\in \mathfrak{g}$, i.e.
$$ A^{\mathfrak{g}} = \bigcap_{x\in \mathfrak{g}} \mathrm{Ker}(d_x).$$
For an arbitrary left $A\# U(\mathfrak{g})$-module $M$ one sets $M^\mathfrak{g} = \bigcap_{x\in \mathfrak{g}} \mathrm{Ann}_M(1\# x).$
As before there are functorial $R$ linear isomorphisms $M^{\mathfrak{g}} \simeq \mathrm{Hom}_{A\#U(\mathfrak{g})}(A,M)$ and in particular $A^\mathfrak{g} \simeq \mathrm{End}_{A\#U(\mathfrak{g})}(A).$
Retractability for $A$ as a left $A\#U(\mathfrak{g})$-module also means here that $A^{\mathfrak{g}}$ is large in $A$ with respect to all those left ideals of $A$ that are stable under alll derivations $d_x$ with $x\in \mathfrak{g}$.

In the case of a single derivation $\partial \in \mathrm{Der}_R(A)$ one considers the trivial Lie algebra $\mathfrak{g}=R$ with zero bracket. The map $d:R\rightarrow \mathrm{Der}_R(A)$ is then given by $r\mapsto r\partial$ for all $r\in R$. Note that the enveloping algebra of the trivial Lie algebra is the polynomial ring $R[x]$ in one variable. Moreover 
$$A\# U(\mathfrak{g}) = A{\otimes_R} R[x] = A[x]$$ is determined by the product:
$$ (1\# x) (a\#1) = a\# x + \partial(a)\# 1\qquad \mbox{ or better} \qquad xa =ax + \partial(a) \qquad \forall a\in A,$$
showing that $A\# R[x] \simeq A[x;\partial]$.

\subsection{Group Actions}
A group $G$ act on an $R$-algebra $A$ by automorphism, which means that there is a homomorphism of groups $\rho:G\rightarrow \mathrm{Aut}_R(A)$ from $G$ to the group of $R$-linear automorphisms of $A$. We denote the image of $g\in G$ under $\rho$ by ${\rho}_g$, although we sometimes write $^ga$ instead of  ${\rho}_g(a)$ for $a\in A$ and $g\in G$. As in the last section, a new algebra can be attached to $G$ and $A$, which is the skew-group ring $A*G$ defined on $A {\otimes_R} R[G]$, where $R[G]$ is the group ring of $G$ over $R$. Alternatively one might consider $A*G$ as the free left $A$-module with basis $\{ \overline{g} \mid g\in G\}$ such that the multiplication is defined as 
$$ a\overline{g}\cdot b\overline{h} = a{\rho}_g(b)\overline{gh}  = a({^gb})\overline{gh} \qquad \forall a,b\in A, \forall g,h\in G.$$
If $G$ is cyclic infinite, i.e. $G=\langle \sigma \rangle$, then $A*G$ is equal to the Laurent skew-polynomial ring $A[x,x^{-1}; \sigma]$ whose underlying space are the Laurent polynomials with coefficients in $A$ and whose multiplication is determined by
$$ x^na = \sigma^n(a) x^n \qquad \forall a\in A, n\in \ZZ.$$
If $G=\langle \sigma\rangle $ is cyclic of order $n$, then $A*G$ is equal to the factor $A[x;\sigma]/\langle x^n-1\rangle$ of the skew-polynomial ring $A[x;\sigma]$. 

Let $G$ be any group acting on $A$ as automorphism. Then $A$ has a left $A*G$-module structure defined by $$a\overline{g} \cdot b = a \rho_g(b) \qquad \forall a,b\in A, g\in G.$$
The left $A*G$-submodules of $A$ are precisely the $G$-stable left ideals of $A$. Let $M$ be any left $A*G$-module. Then the submodule of fixed elements of $M$ is  $$M^G=\{m\in M \mid \forall g\in G: \: g\cdot m = m.\}$$
Moreover one has again $R$-linear isomorphisms for any left $A*G$-module $M$:
\begin{equation}\psi_M: \mathrm{Hom}_{A*G}(A,M) \longrightarrow M^G  \qquad \mbox{ given by } \qquad f \mapsto  (1)f \qquad \forall f\in\mathrm{Hom}_{A*G}(A,M).
\end{equation}
with inverse map
\begin{equation}\psi_M^{-1}: M^G \longrightarrow \mathrm{Hom}_{A*G}(A,M) \qquad \mbox{ given by } \qquad m \mapsto [a\mapsto a\cdot m] \qquad \forall m\in M^G.\end{equation}
In particular $\psi_A:\mathrm{End}_{A*G}(A) \simeq A^G$ is an isomorphism of $R$-algebras and the bijections $\psi_M$ are left $A^G$-linear maps.

\begin{lem}
$A$ is a retractable $A*G$-module if and only if $A^G$ is large in $A$, i.e. $A^G$ intersects all non-trivial $G$-stable left ideals of $A$ non-trivially.
\end{lem}

The existence of non-trivial fixed elements in $G$-stable left ideals reduces the study of the structure of $G$-stable left ideals of $A$ to the study of left ideals of $A^G$. The following result is a classical theorem in the study of group action:

\begin{thm}[Bergman-Isaacs, 1973; Kharchenko, 1974]\label{Bergman-Isaacs}
Let $G$ be a finite group of order $n$ acting on an  $R$-algebra  $A$. Assume that one of the following conditions is verified:
\begin{enumerate}
\item $A$ is $n$-torsionfree and does not contain any non-zero nilpotent $G$-stable ideal or 
\item $A$ is reduced, i.e. does not contain any nilpotent element.
\end{enumerate}
Then $A$ is retractable as left $A*G$-module.
\end{thm}
For the proof of (1) see \cite{BergmanIsaacs, Puczylowski} for the proof of (2) see \cite{Kharchenko}.

\subsection{Involutions}\label{involution}
Let $A$ be an $R$-algebra. An involution of $A$ is an $R$-linear map $\ast: A\rightarrow A$ with $a\mapsto a^\ast$ that is an anti-algebra homomorphism and has order $2$, i.e. $(ab)^\ast = b^\ast a^\ast$ and $(a^\ast)^\ast=a$ for all $a,b\in A$. An element $a\in A$ is called symmetric (respectively anti-symmetric) with respect to $\ast$ if $a^\ast=a$ (respectively. $a^*=-a$). Ideals that are stable under the involution $\ast$ are called $\ast$-ideals.
Consider the subalgebra $B$ of $\mathrm{End}_R(A)$ generated by $\ast$ and all left multiplications $l_a$ for $a\in A$, i.e. 
$$B=\langle \{\ast\}\cup \{l_a \mid a\in A\}\rangle \subseteq \mathrm{End}_R(A).$$
Note that for any $a,b\in A$ one has
$$ r_a(b) = ba = (a^\ast b^\ast)^\ast = (\ast \circ l_{a^\ast} \circ \ast)(b).$$
Hence $r_a = \ast \circ l_{a^\ast} \circ \ast \in B$.  It is clear that $A$ becomes a left $B$-module by simply applying  $f \in B\subseteq \mathrm{End}_R(A)$, i.e. for any $f\in B$ and $a\in A$ set $f \cdot a := f(a)$. The left $B$-submodules of $A$ are stable under left and right multiplications $l_a$ and $r_a$ for any $a\in A$ and hence are two-sided ideals of $A$. Moreover they are stable under $\ast$. On the other hand any $\ast$-ideal is also a left $B$-submodule.

The algebra $B$ can be seen as a factor of a skew group algebra. Let $A^e=A{\otimes_R} A^{op}$ be the enveloping algebra of $A$ and consider the map
$\alpha: A^e \rightarrow A^e$ defined by $\alpha(a\otimes b) = b^\ast \otimes a^\ast$ for all $a,b\in A$. The map $\alpha$ is an automorphism of $A^e$, because for any $a,b,c,d\in A$:
$$\alpha \left((a\otimes b)(c\otimes d)\right) = \alpha\left( ac \otimes db  \right) = (db)^\ast  \otimes (ac)^\ast  = b^\ast d^\ast \otimes c^\ast a^\ast = 
\left( b^\ast \otimes a^\ast\right)\left( d^\ast \otimes c^\ast\right) = \alpha(a\otimes b)\alpha(c\otimes d).$$
As $\alpha$ is its own inverse it is an automorphism of order $2$. Let $G=\langle \alpha\rangle = \{id, \alpha \}$ and consider the (surjective) map 
$\psi: A^e * G \rightarrow B$ given by 
$(a\otimes b)\otimes id + (c\otimes d) \otimes \alpha \mapsto l_a\circ r_b  + l_c\circ r_d \circ \ast$ for all $a,b,c,d\in A$. Then $\psi$ is an algebra homomorphism. The calculations are easy but tedious and will be illustrated on the example of the product of $(1\otimes 1)\otimes  \alpha $ and $(a\otimes b)\otimes id$.
Note first that  for any $x\in A$:
$$l_{b^\ast}\circ r_{a^\ast}\circ \ast(x) = b^\ast x^\ast a^\ast = (axb)^\ast = \ast\circ l_a\circ r_b (x).$$
Hence
$$ \psi\left(( 1\otimes 1 \otimes \alpha)(a\otimes b \otimes id )\right) = \psi\left(b^\ast\otimes a^\ast \otimes \alpha\right) = l_{b^\ast}\circ r_{a^\ast}\circ \ast = \ast \circ l_a \circ r_b = \psi\left (1\otimes 1 \otimes \alpha)(a\otimes b \otimes id)\right).$$
Thus $B$ is a factor algebra of $A^e*G$.
For any left $A^e*G$-module $M$ one defines the submodule of central symmetric elements as 
$$Z(M;\ast) = Z(M)\cap M^G =  \{ m\in Z(M) \mid \alpha \cdot m = m\}.$$
For $M=A$ one obtains the central symmetric elements of $A$, i.e. $Z(A;\ast)=\{a\in Z(A)\mid a^\ast=a\}$.
Again one has $R$-isomorphisms for any left $A^e* G$-module $M$:
\begin{equation}\psi_M: \mathrm{Hom}_{A^e* G}(A,M) \longrightarrow Z(M;\ast)  \qquad \mbox{ given by } \qquad f \mapsto  (1)f \qquad \forall f\in\mathrm{Hom}_{A^e*G}(A,M).
\end{equation}
with inverse map
\begin{equation}\psi_M^{-1}: Z(M; \ast)  \longrightarrow \mathrm{Hom}_{A^e*G}(A,M) \qquad \mbox{ given by } \qquad m \mapsto [a\mapsto a\cdot m] \qquad \forall m\in Z(M;\ast).\end{equation}
In particular $\psi_A:\mathrm{End}_{A^e*G}(A) \simeq Z(A;\ast)$ is an isomorphism of $R$-algebras and the bijections $\psi_M$ are left $Z(A; \ast)$-linear maps.

\begin{lem} $A$ is a retractable $A^e*G$-module if and only if every non-zero $\ast$-ideal contains a non-zero central symmetric element.
\end{lem}

Using Rowen's theorem we have the following:

\begin{cor} Let $\ast$ be an involution of an $R$-algebra $A$. If $A$ is a semiprime PI-algebra, then $A$ is a retractable $A^e*G$-module.
\end{cor}
\begin{proof}
Let $I$ be a non-zero $\ast$-ideal. By Rowen's Theorem \ref{RowenTheorem}, $I$ contains a non-zero central element, say $a\in I$. Since $I$ is $\ast$-stable, $a^\ast \in I$.
Thus $a + a^\ast$ is a central symmetric element of $I$. If $a+a^\ast=0$, then $a^\ast = -a$ and $a^2$ is a central symmetric element as $(a^2)^\ast = (-a)^2 = a^2.$  Note that $a^2\neq 0$ as $a$ is non-zero and central and $A$ is semiprime.
 
\end{proof}

\section{Open Problems}
If a group $G$ acts on an algebra $A$ by automorphisms, then $A$ becomes a module over the skew group ring $A*G$ as well as over the group algebra $R[G]$. If a Lie algebra $\g$ acts on $A$ by derivations, then $A$ becomes a module over $A\# U(\g)$ as well as over the universal enveloping algebra $U(\g)$. Both algebras $R[G]$ and $U(\g)$ are examples of Hopf algebras and the constructions $A*G$ respectively $A\#U(\g)$ are so-called smash products. A Hopf algebra $H$ is an algebra such that  there exist algebra maps
$\Delta: H\rightarrow H\otimes H$ (called the comultiplication of $H$) and $\epsilon:H\rightarrow R$ (called the count of $H$) such that the following diagrams commute:

\[\xymatrix{
H \ar@{->}[r]^\Delta \ar@{->}[d]_\Delta & H\otimes H\ar@{->}[d]^{\Delta\otimes 1}\\
H\otimes H \ar@{->}[r]_{1\otimes \Delta} & H\otimes H \otimes H}\qquad
\xymatrix{R\otimes H & H\otimes H \ar@{->}[l]_{\epsilon \otimes 1}\ar@{->}[r]^{1\otimes \epsilon} & H\otimes R \\
 & H \ar@{->}[lu]^{\simeq}  \ar@{->}[ru]_{\simeq}  \ar@{->}[u]^\Delta & 
}\]

For an element $h\in H$ its comultiplication $\Delta(h)$ is an element of $H\otimes H$. It is common to use the so-called Sweedler's notation enumerating symbolically the first and second tensorand by writing $\Delta(h)= \sum_{(h)} h_1\otimes h_2 \in H\otimes H$.

The ring of $R$-linear endomorphisms of $\mathrm{End}_R(H)$ of a Hopf algebra $H$ has another ring structure as the usual given by the convolution product which associates to two endomorphisms $f,g$ of $H$ the endomorphisms $f\ast g = \mu\circ (f\otimes g) \circ \Delta$ where $\mu$ denotes the multiplication of $H$. To obtain a Hopf algebra one also requires that the identity map has an inverse in $\mathrm{End}_R(H)$ with respect to the convolution product.  This inverse is usually denoted by $S$ and called the antipode of $H$. Equivalently there should exist an endomorphism $S$ satisfying
$$ \mu \circ (id\otimes S)\circ \Delta = \eta\circ \epsilon = \mu \circ (S\otimes id) \circ \Delta$$
where $\eta:R\rightarrow H$ denotes the map $r\mapsto r1_H$ for all $r\in R$.
A Hopf algebra $H$ acts on an $R$-algebra $A$ if $A$ is a left $H$-module and an algebra in the category of left $H$-modules. The later means that the multiplication  $m:A{\otimes_R} A\rightarrow A$ and the unit map $R\rightarrow A$ with $r\mapsto r1_A$ are maps of left $H$-modules. Note that due to the comultiplication the category of left $H$-modules is closed under tensor products, i.e. it is a tensor category. For left $H$-modules $N$ and $M$, elements $x\in N$ and $y\in M$ and $h\in H$ with $\Delta(h)=\sum_{(h)} h_1\otimes h_2 \in H\otimes H$ one sets 
$ h\cdot (x\otimes y ) = \sum_{(h)} (h_1 \cdot x)\otimes (h_2 \cdot y).$  The base ring $R$ becomes a left $H$-module by $h\cdot r = \epsilon(h)r$ for all $h\in H, r\in R$. Hence for $H$ to act on $A$, $A$ has to be a left $H$-module and the following conditions have to be fulfilled for all $a,b\in A$ and $h\in H$.
$$ h\cdot (ab) = \sum_{(h)} (h_1\cdot a ) ( h_2\cdot b) \qquad \mbox{ and } \qquad h\cdot 1_A = \epsilon(h)1_A.$$
The smash product of $A$ and $H$ is denoted by $A\# H$ and defined on the tensor product $A {\otimes_R} H$ with multiplication given by
$$ (a\otimes h)\cdot (b\otimes g) = \sum_{(h)} a(h_1\cdot b) \otimes h_2g \qquad \forall a,b\in A, h, g \in H.$$
Then $A$ becomes a cyclic left $A\# H$-module by the action $(a\otimes h)\bullet b = a(h\cdot b)$ for all $a,b\in A$, $h\in H$.
For any left $A\# H$-module $M$ one defines the submodule of $H$-invariants of $M$ as 
 $$M^H = \{ m\in M \mid h\cdot m = \epsilon(h)m \: \forall h\in H\}.$$
 As in the previous sections one has a for any left $A\# H$-module $M$ canonical maps:
\begin{equation}\psi_M: \mathrm{Hom}_{A\# H}(A,M) \longrightarrow M^H \qquad \mbox{ given by } \qquad f \mapsto  (1)f \qquad \forall f\in\mathrm{Hom}_{A\# H}(A,M).\end{equation}
with inverse map
\begin{equation}\psi_M^{-1}: M^H \longrightarrow \mathrm{Hom}_{A\# H}(A,M) \qquad \mbox{ given by } \qquad m \mapsto [a\mapsto a\cdot m] \qquad \forall m\in M^H.\end{equation}
In particular $\psi_A:\mathrm{End}_{A\# H}(A) \simeq A^H$ is an isomorphism of $R$-algebras and the bijections $\psi_M$ are left $A^H$-linear maps.

For a group $G$ and its group algebra $H=R[G]$,  the Hopf algebra structure of $H$ is given by the comultiplication $\Delta(g)=g\otimes g$, the counit $\epsilon(g)=1$ and the antipode $S(g)=g^{-1}$, for all $g\in G$. 
The group algebra $H$  acts on $A$ if there is a module action $H\otimes A \rightarrow A$, say by  $h\otimes a \mapsto h\cdot a$, for all $h\in H, a\in A$ and the two conditions from above are satisfied, i.e.
$$ g\cdot (ab) = (g\cdot a)(g\cdot b) \qquad \mbox{ and } \qquad g\cdot 1_A = \epsilon(g)1_A = 1_A \qquad \forall g\in G.$$
Hence if one denotes by $\alpha_g$ the map $a\mapsto \alpha_g(a)=g\cdot a$, then one sees from this two conditions that $\alpha_g$ is an endomorphism of rings. Since $G$ is a group and $A$ is supposed to be a left $H$-module, $\alpha_{g^{-1}}= \alpha_g^{-1}$ for all $g\in G$. It is easy to see that $A*G$ equals the smash product $A\# R[G]$.

For a Lie algebra $\g$ and its universal enveloping algebra  $H=U(\g)$ , the Hopf algebra structure of $H$ is given by the comultiplication 
 $\Delta(x)=1\otimes x + x\otimes 1$, the counit $\epsilon(x)=0$ and the antipode $S(x)=0$, for all $x\in \g$. 
The Hopf algebra $H$  acts on $A$ if there is a module action $H\otimes A \rightarrow A$, say by  $h\otimes a \mapsto h\cdot a$, for all $h\in H, a\in A$ and the two conditions from above are satisfied, i.e.
$$ x\cdot (ab) = (1\cdot a)(x\cdot b) + (x\cdot a)(1\cdot b) = a(x\cdot b) + (x\cdot a)b \qquad \mbox{ and } \qquad x\cdot 1_A = 0 \qquad \forall x\in \g.$$
Hence if one denotes by $\partial_x$ the map $a\mapsto \partial_x(a)=x\cdot a$, for $x\in \g$, then one sees from this two conditions that $\partial_x$ is a derivation of $A$. Hence the Hopf algebra $U(\g)$ acts on $A$ if each element acts as a derivation on $A$. The converse also holds. In particular if $\partial$ is a single derivation on $A$, $\g=R$ is the trivial Lie algebra and $U(\g)=R[x]$ is its universal enveloping algebra, one obtains an action of $H=R[x]$ on $A$ by setting for any polynomial $f(x)=\sum_{i=0}^n r_ix^i$ and element $a\in A$: $f(x) \cdot a  = \sum_{i=0}^n r_i \partial^i(a).$
Looking at the smash product $A\# R[x]$ one sees that the multiplication coincides with that of the differential operator ring $A[x;\partial]$, because
$$ (1_A \otimes x)(b\otimes 1_H) = 1_H (1_H\cdot b) \otimes x1_H + 1_H (x \cdot b) \otimes 1_H 1_H = b\otimes x + \partial(b) \otimes 1_H.$$
Identifying $A{\otimes_R} R[x]$ with $A[x]$ as $R$-modules, we obtain our usual multiplication rule.

\subsection{Cohen's problem}
Miriam Cohen asked in 1985 whether the smash product $A\# H$ is a semiprime ring provided $A$ is semiprime and $H$ is a semisimple Hopf algebra acting on $A$ (see \cite{cohen85}). The questions is still open up to my knowledge. The semisimple condition on $H$ implies that $A$ is a projective $A\# H$-module. Hence the map $\varphi: A\# H\rightarrow A$ with $a\# h \mapsto a\epsilon(h)$ splits by the map $\psi:A\rightarrow A\# H$ such that $e=(1_A)\psi$  is an idempotent in $(A\# H)^H$. Moreover for any $H$-stable left ideal $I$ of $A$ one has that $(I)\psi = (I\# 1)e$ is a left ideal of $A\# H$ isomorphic to $I$.
If  $A\# H$ is semiprime and $A$ projective as $A\# H$-module then 
$(I\# 1)e(I\# 1)e$ would be non-zero for any non-zero $H$-stable left ideal $I$ of $A$. Hence $0\neq e(I\# 1)e = e(I)\psi = (e\bullet I)\psi$ shows that $e\bullet I$ is non-zero. As $e\in (A\# H)^H$ one also has $0\neq e\bullet I \in A^H\cap I \simeq \mathrm{Hom}_{A\# H}(A,I)$. 

\begin{cor} If $A\# H$ is semiprime and $A$ is projective as left $A\# H$-module, then $A$ is a retractable left $A\# H$-module.
\end{cor}

Of course this is true much more generally, namely for torsionless modules over semiprime rings, due to Amitsur (see \cite[Theorem 27, Corollary 21]{amitsur}):

\begin{lem}[Amitsur] Any left $A$-module over a semiprime ring $A$ that can be embedded into a direct product of copies of $A$ is retractable.
\end{lem}

Hence if Cohen's question would have a positive answer, then $A$ semiprime and $H$ being semisimple acting on $A$ would imply that $A$ is a retractable $A\# H$-module.

\begin{question} Let $H$ be a semisimple Hopf algebra acting on a semiprime algebra $A$. Is $A$ a retractable left $A\# H$-module, or in other words, does every non-zero $H$-stable left ideal intersect $A^H$ non-trivially ?
\end{question}

Recalling Bergman and Isaacs Theorem \ref{Bergman-Isaacs} we see that an answer to the question above would generalise their Theorem and would be "half way" towards a positive answer to Cohen's question. For more on Cohen's question I would like to refer to my recent survey \cite{Lomp_Survey}.

\subsection{Primness of endomorphism rings}

There are many ways to carry the idea of a prime ring to modules. Bican et al. in \cite{bican} defined a product on the lattice of submodules $\cL$ of a left $A$-module $M$, by defining for any $N,K \in \cL$: 
$$ N\ast K := N\mathrm{Hom}_A(M,K) = \sum_{f:M\rightarrow K} (N)f .$$
Together with this product, $\cL$ becomes a partially ordered groupoid, i.e. $\cL$ is a partially ordered set by inclusion and the binary operation  $\ast$ satisfies  $N\ast K \subseteq L \ast K$ and $K\ast N \subseteq K\ast L$, whenever  $K,L,N\in\cL$ and $N\subseteq L$. Note that in general the $\ast$-product is not associative. 
The notion of a prime element is naturally carried over to any partially ordered groupoid (see \cite{Birkhoff}), i.e. $P\in \cL$ is prime if $N\ast K \subseteq P$ implies $N\subseteq P$ or $K\subseteq P$, and the module $M$ is called $\ast$-prime if $0$ is a prime element in the partially ordered groupoid $(\cL, \ast)$ (see \cite{bican, lomp_prime}). By definition, if $M$ is $*$-prime, then $0\neq M\ast N =M\mathrm{Hom}_A(M,N)$ for all $0\neq N\subseteq M$. Hence $M$ is retractable. On the other hand if $M$ is retractable and $\mathrm{End}_A(M)$ is a prime ring, then it is not difficult to see that $M$ is $\ast$-prime (see \cite[4.1]{lomp_prime}). However it had been left open whether in general a $\ast$-prime module must have a prime endomorphism ring.
\begin{question}
Does a $\ast$-prime module have a prime endomorphism ring ?
\end{question}
 Several positive results were obtained in \cite{BaziarLomp}. In particular if the endomorphism ring of $M$ is commutative, then the answer is yes (see \cite[2.3]{BaziarLomp}). This applies in particular to the case where $M=A$ is an algebra considered as a left $A^e$-module. As $\mathrm{End}_{A^e}(A)\simeq Z(A)$ is the centre of $A$, we have that $A$ is $\ast$-prime as left $A^e$-module if and only if $A$ has a large centre which is an integral domain. An analogous statement hold for algebras with involution.
 Furthermore it has been shown in \cite{BaziarLomp} that if a module $M$ is semi-projective or not singular, then the answer is also yes. Thus if a semisimple Hopf algebra $H$ acts on $A$, then $A$ is projective as left $A\# H$-module and $A$ is a $\ast$-prime left $A\# H$-module if and only if $A^H$ is large in $A$ and also a prime ring. This applies also to the case where $H=R[G]$ with $G$ a finite group acting on $A$ such that $n=|G|$ is invertible in $R$.
Another instance where the general module theoretic result can be applied is in case of a locally nilpotent derivation  $\partial$ on $A$ over a field $R$ of characteristic $0$ such that there exists $x\in A$ with $\partial(x)=1$. Then $A$ is a self-projective $A[x;\partial]$-module and thus $A$ is $\ast$-prime as left $A[x;\partial]$-module if and only if $A^\partial$ is a prime ring (since the locally nilpotency of $\partial$ implies the retractability of $A$).
However in general it is unknown whether $A^\partial$ is always prime if $A$ is a $\ast$-prime left $A[x;\partial]$-module.

\begin{question} Let $\partial$ be any derivation on $A$. Is it true that $A^\partial$ is a prime ring provided $A$ is a $\ast$-prime $A[x;\partial]$-module?
\end{question}

Zelmanowitz' weakly compressible modules $M$ are precisely those with $0$ being a semiprime element in the partially ordered groupoid $(\cL, \ast)$, i.e. if $N\ast N \subseteq 0$, then $N=0$. This means that for any non-zero submodule $N$ of $M$ there exists a homomorphism $f:M\rightarrow N$ with $(N)f\neq 0$. Clearly the notion of weakly compressible modules generalise the notion of $\ast$-prime modules as well as the notion of compressible modules. Furthermore weakly compressible modules are obviously retractable.
Alternatively a module $M$ is called \emph{semiprime} if any essential submodule of $M$ cogenerates $M$. Having in mind that $\ast$-prime modules $M$  can be characterised by the property that any non-zero submodule $N$ of $M$  cogenerates $M$, one sees that also semiprime modules are a generalisation of $\ast$-prime modules. It is not difficult to see that weakly compressible modules are semiprime (see \cite[5.5]{lomp_prime}) and it is an open question whether the converse holds:

\begin{question}
Is a semiprime module weakly compressible?
\end{question}

This question has been considered in \cite{Vedadi} and remotely also in \cite{SmithVedadi2} and it seems that it is important to know whether 
semiprime modules are retractable. Hence we might ask:

\begin{question}
Are semiprime modules retractable? 
\end{question}

What can be said about the cases mentioned above, e.g. if $M=A$ and $H$ is a Hopf algebra acting on $A$ ?

\subsection{Zelmanowitz' problem}
A \emph{critically compressible} module is a compressible module that cannot be embedded into any of its proper factor modules.
The following question arose in Zelmanowitz' papers \cite{Zelmanowitz, Zelmanowitz2}:
\begin{question}\label{ZQuest} Is a  compressible uniform module whose nonzero endomorphisms are monomorphisms a critically compressible module?
\end{question}
Attempts to answer this question have been made in \cite{RodriguesSantana} where it has been shown that the hypothesis of the
 question are equivalent to $M$ being a uniform retractable module whose endomorphism ring is a domain. Furthermore the following result has been shown in \cite[Proposition 3.1]{BorgesLomp}.

\begin{prop}[Borges-Lomp, 2011] 
Let $M$ be a left $A$-module with endomorphism ring $S$ and self-injective hull $\widehat{M}$.
Then  $M$ is critically compressible if and only if $M$ is retractable, $S=\mathrm{End}(M)$ is a left Ore domain and 
$\mathrm{End}(\widehat{M}) = \mathrm{Frac}(S)$ is the left division ring of fractions of $S$.
\end{prop}

Thus a negative answer to question \ref{ZQuest} could be obtained through an example of a uniform retractable module whose endomorphism ring is a domain, but not a left Ore domain. We finish by asking this question for our algebra situations above. All mentioned examples were of the following form (see also \cite{Lomp_central}):
Let $A$ be an algebra over $R$. Let $B$ be a subalgebra of $\mathrm{End}_R(A)$ that contains all left multiplications $l_a$ for $a\in A$. Then $A$ becomes a left $B$-module by evaluating endomorphisms, i.e. $b\in B, a\in A: b\bullet a := b(a)$ and moreover $A$ can be considered a subalgebra of $B$ by $a\mapsto l_a$ for all $a\in A$. The map $\alpha: B\rightarrow A$ given by $b\mapsto (b)\alpha=b\bullet 1$ for all $b\in B$ is a surjective map of left $B$-modules, i.e. $A$ is a cyclic left $B$-module, and furthermore the inclusion $a\mapsto l_a$ lets $\alpha$ split as left $A$-modules, because $(l_a)\alpha=l_a(1)=a$ for all $a\in A$. Any left $B$-module $M$ can be naturally considered a left $A$-module and there exists a map
\begin{equation}\label{psi_general}
\psi_M: \mathrm{Hom}_{B}(A,M) \longrightarrow M \qquad \mbox{ given by } \qquad f \mapsto  (1)f \qquad \forall f\in\mathrm{Hom}_{B}(A,M).
\end{equation}
The image of $\psi_M$ can be identified with the subset 
$$\mathrm{Im}(\psi_M) = \{m \in M \mid b\bullet m = (b)\alpha m = (b\bullet 1)m\} =: M^B$$
which we call the submodule of $B$-invariants of $M$. In particular $A^B\simeq \mathrm{End}_B(A)$  is naturally isomorphic to the endomorphism ring of $A$ as left $B$-module. As in the cases above one also has an inverse of $\psi_M$ which is:
\begin{equation}\psi_M^{-1}: M^B \longrightarrow \mathrm{Hom}_{B}(A,M) \qquad \mbox{ given by } \qquad m \mapsto [a\mapsto a\bullet m] \qquad \forall m\in M^B.\end{equation}
As before one has that $A$ is a retractable left $B$-module if and only if $A^B$ is large in $A$, i.e. $A^B$ intersects all $B$-stable left ideals of $A$.

\begin{question}
Suppose that $A^B$ is a domain and large in $A$. If $A$ is a left uniform $B$-module, does it follow that $A^B$ is a left Ore domain.
\end{question}

In case $B=M(A)$ as in subsection \ref{bimodule} respectively  in case $B$ is the subalgebra of $\mathrm{End}_R(A)$ generated by the left multiplication and an involution $\ast$ as in subsection \ref{involution}, one obtains that $A^B=Z(A)$ respectively $A^B=Z(A;\ast)$ is commutative. Since a commutative domain is an Ore domain, the question above is obviously answered. However in case of a group action or an action by a derivation, $A^B$ might be  non-commutative and raises the following questions:

\begin{question} Let $\partial$ be a derivation on $A$ such that $A^\partial$ is a domain which is large in $A$. Is $A^\partial$ left Ore if  $A$ is a uniform left $A[x;\partial]$-module?
\end{question}
 Let $\sigma$ be an automorphism of $A$. Then $ A^G = \{a\in A \mid \sigma(a)=a\}=:A^\sigma$, where $G=\langle \sigma \rangle$.
 
\begin{question} Let $\sigma$ be an automorphism such that  $ A^\sigma$ is a domain which is a large in $A$.
Is $A^\sigma$ left Ore if $A$ is a uniform left $A[x;\sigma]$-module (respectively left $A[x^{\pm 1}; \sigma]$-module)?
\end{question}

\end{document}